\documentclass[12pt,a4paper,twoside]{article}
\usepackage[english]{babel}

\usepackage{fancyhdr}
\usepackage{indentfirst} 
\usepackage{newlfont}
\usepackage{verbatim}
\usepackage{enumerate}

\usepackage{amsthm}  
\usepackage{amssymb}  
\usepackage{amscd}
\usepackage{amsmath} 
\usepackage{latexsym} 
\usepackage{amsfonts}
\usepackage{amsbsy}   
\usepackage{mathrsfs} 
\usepackage{arcs}
\usepackage{yhmath}

\usepackage{graphicx}
\usepackage{color}
\usepackage{subfigure}

\theoremstyle{plain}
\newtheorem{teo}{Theorem}
\newtheorem{prop}[teo]{Proposition}
\newtheorem{cor}[teo]{Corollary}
\newtheorem{lem}[teo]{Lemma}
\theoremstyle{definition}

\theoremstyle{remark}
\newtheorem{oss}[teo]{Remark}

\newcommand{\rp}[1]{\ensuremath{\mathbb{RP}^{#1}}}
\newcommand{\s}[1]{\ensuremath{\mathbf{S}^{#1}}}

\begin{document}

\title{On knots and links in lens spaces}

\author{Alessia Cattabriga, Enrico Manfredi, Michele Mulazzani}

\maketitle

\begin{abstract}

In this paper we study some aspects of knots and links in lens spaces. Namely, if we consider lens spaces as quotient of the unit ball $B^{3}$ with suitable identification of boundary points, then we can project the links on the equatorial disk of $B^{3}$, obtaining a regular diagram for them.
In this contest, we obtain  a complete finite set of Reidemeister type moves establishing equivalence, up to ambient isotopy,  a Wirtinger type presentation for the fundamental group of the complement of the link and a diagrammatic method giving the first homology group.
We also compute Alexander polynomial and twisted Alexander polynomials of this class of links, showing their correlation with Reidemeister torsion.\\
\\ {{\it Mathematics Subject
Classification 2010:} Primary 57M25, 57M27; Secondary  57M05.}\\

Work performed under the auspices of G.N.S.A.G.A.
of C.N.R. of Italy and supported by M.U.R.S.T., by the University
of Bologna, funds for selected research topics.

\end{abstract}

\begin{section}{Introduction}

Knot theory is a widespread branch of geometric topology, with many
applications to theoretical physics, chemistry and biology. The
mainstream of this research have been concentrated for more than one
century in the study of knots/links in the 3-sphere, which is the
simplest closed 3-manifolds, and where the theory is completely
equivalent to the one in the familiar space $R^3$. That study was maily
conducted by the use of regular diagrams, which are suitable projection
of the knot/link in a disk/plane. In this way the 3-dimensional
equivalence problem in translated in a 2-dimensional equivalence problem
of diagrams. Reidemeister proved that two knots/links are equivalent if
any of their diagrams can be connected by a finite sequence of three
local moves, called Reidemeister moves. Diagrams also helps to obtain
invariants as the fundamental group of the exterior of the link
(also called group of the link), via Wirtinger
theorem, while the homology groups, as well as higher homotopy groups,
are not relevant in the theory. From the fundamental group other
important invariant as Alexander polynomials (classical and twisted) have been obtained, while
from the diagram state sum type invariant derive, as Jones polinomials and quandle invariants.

In the last two decades, studies on knots/links have been generalized in
more complicated spaces as solid torus (see \cite{Be}, \cite{Ga1},
\cite{Ga2}), or lens spaces, which are the simplest closed
3-manifolds different from the 3-sphere.
Particuarly important are the class of $(1,1)$-knots (knots in either
$\s3$ or a lens space, also called genus one 1-bridge knots) intensively
studied by many authors (see \cite{CM}, \cite{CK}, \cite{Fu}, \cite{Ha},
\cite{MS}, \cite{Wu}).

In 1991, Drobotukhina introduced diagrams and moves for knots and links
in the projective space, which is a special case of lens space,
obtaining in this way an approach to compute a Jones type invariant for
these links (see \cite{Dr}).
More recently, Huynh and Le in \cite{HL} obtained a formula for the
computation of the twisted Alexander polynomial for links in the
projective space.

In this paper we extend some of those results for knots/links in the
whole family of lens spaces. Our approach use the model of lens spaces
obtained by suitable identification on the boundary of a 3-ball
described in Section 2, where a concept of regular projection and relative
diagrams for the link is defined. In Section 3 we show that the
equivalence between links in lens spaces can be translated in
equivalence between diagrams, via a finite sequence of seven type of
moves, generalizing the Reidemeister ones. In Section 4 a Wirtinger type
presentation for the group of the link is given. In this contest the
homology group are not abelian free groups (as in \s3), since a torsion part
appears, and in Section 5 a method to compute that directly from the
diagram is given. In Section 6 we deal with the twisted Alexander polynomials of these links, finding different properties and exploiting the connection with the Reidemeister torsion.

\end{section}

%============================================================

\begin{section}{Diagrams}

In this paper we work in the \emph{Diff} category (of smooth manifolds and smooth maps). Every result also holds in the \emph{PL} category, and in the \emph{Top} category if we consider only tame links.

A \emph{link} $L$ in a closed $3$-manifold $M^{3}$ is a 1-dimensional submanifold \hbox{$L\subset M^{3}$.} Obviously, $L$ is homeomorphic to $\nu$ copies of $\s{1}$. When $\nu=1$ the link is called a \emph{knot}. Two links $L',L''\subset M^{3}$ are called \emph{equivalent} if there exists an ambient isotopy $H: M^{3} \times [0,1] \rightarrow M^{3}$ such that $h_{1}(L')=L''$, where $h_t(x)=H(x,t)$. 

Consider the unit ball $B^{3}=\{(x_{1},x_{2},x_{3}) \in \mathbb{R}^{3} \ | \ x_{1}^{2}+x_{2}^{2}+x_{3}^{2}\leqslant1\}$ and let $E_{+}$ and $E_{-}$ be respectively the upper and the lower closed hemisphere of $\partial B^{3}$. Call $B^{2}_{0}$ the equatorial disk, defined by the intersection of the plane $x_{3}=0$ with $B^{3}$, and label with $N$ and $S$ respectively the "north pole" $(0,0,1)$ and the "south pole" $(0,0,-1)$ of $B^{3}$.

If $p$ and $q$ are two coprime integers such that $ 0 \leqslant q < p$, 
let \mbox{$g_{p,q}: E_{+} \rightarrow E_{+}$} be the rotation of $2 \pi q /p$ around the $x_{3}$-axis, as in Figure~\ref{L(p,q)}, and \hbox{$f_{3}: E_{+} \rightarrow E_{-}$} be the reflection with respect to the plane $x_{3}=0$.
The \emph{lens space} $L(p,q)$ is the quotient of $B^{3}$ by the equivalence relation on $\partial B^{3}$ which identifies $x \in E_{+}$ with $f_{3} \circ g_{p,q} (x) \in E_{-}$. We denote by \mbox{$F: B^{3} \rightarrow L(p,q)=B^{3} / \sim$} the quotient map. Note that on the equator $\partial B^{2}_{0}=E_{+} \cap E_{-}$ each equivalence class contains $p$ points. 

\begin{figure}[h!]                      
\begin{center}                         
\includegraphics[width=9.2cm]{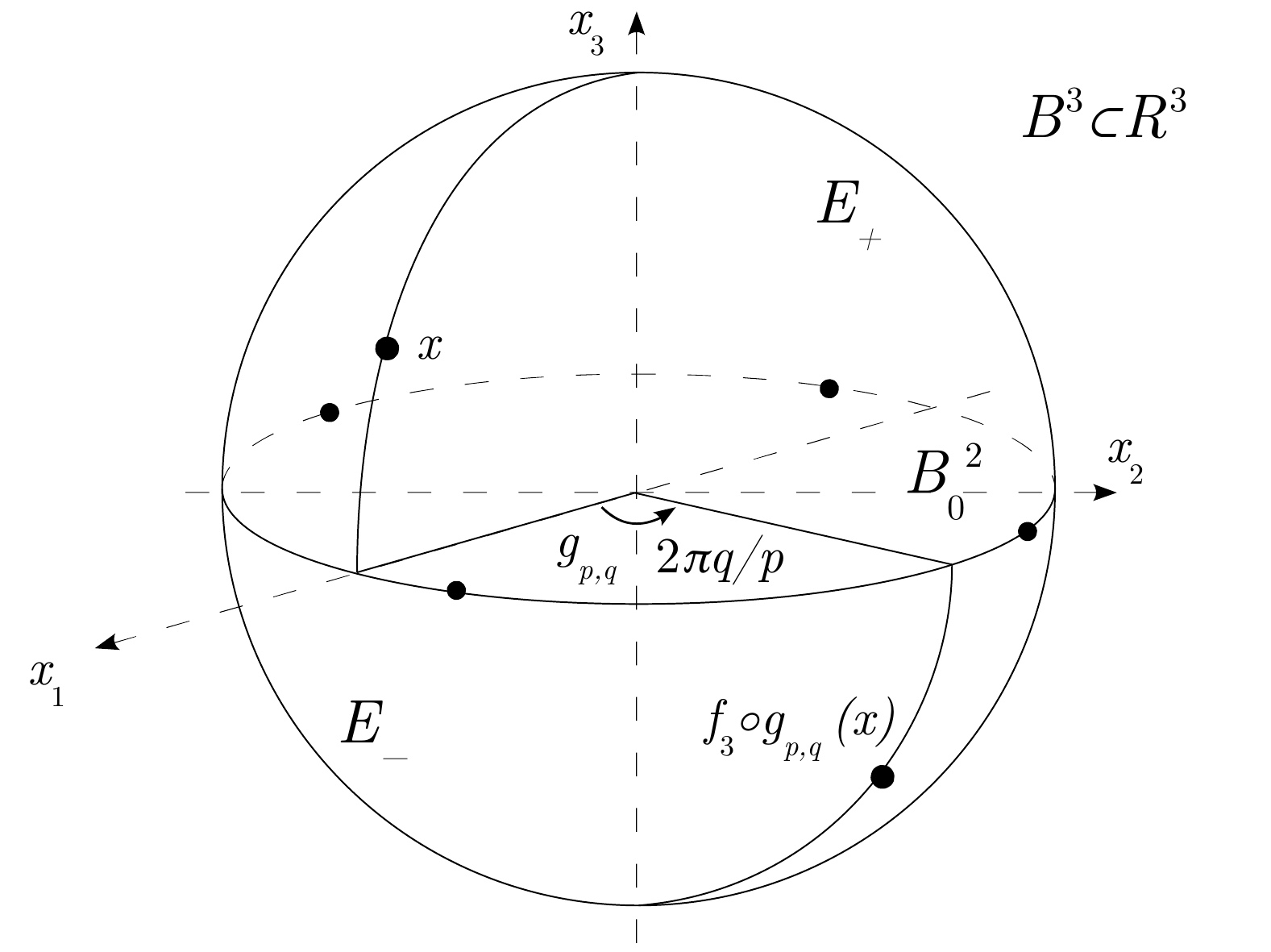}
\caption[legenda elenco figure]{Representation of $L(p,q)$.}\label{L(p,q)}
\end{center}
\end{figure}

It is easy to see that $L(1,0)\cong \s{3}$ since $g_{1,0}=\text{Id}_{E_{+}}$. 
Furthermore, $L(2,1)$ is $ \rp{3}$, since the above construction gives the usual model of the projective space where opposite points on the boundary of $B^3$ are identified.

In the following we improve the definition of diagram for links in lens spaces given by Gonzato \cite{GM}.
Assume $p>1$, since $L(1,0)\cong \s{3}$ is the classical case. Let $L$ be a link in $L(p,q)$ and consider $L'=F^{-1}(L)$. By moving $L$ via a small isotopy in $L(p,q)$, we can suppose that:
\begin{enumerate}
\item[i)] $L'$ does not meet the poles $N$ and $S$ of $B^{3}$;
\item[ii)] $L' \cap \partial B^{3}$ consists of a finite set of points;
\item[iii)] $L'$ is not tangent to $\partial B^3$;
\item[iv)] $L' \cap \partial B^{2}_{0} = \emptyset$.\footnote{The small isotopy that allows $L'$ to avoid the equator $\partial B_{0}^{2}$ is depicted in Figure~\ref{complex}.}
\end{enumerate}

\begin{figure}[h]                      
\begin{center}                         
\includegraphics[width=14cm]{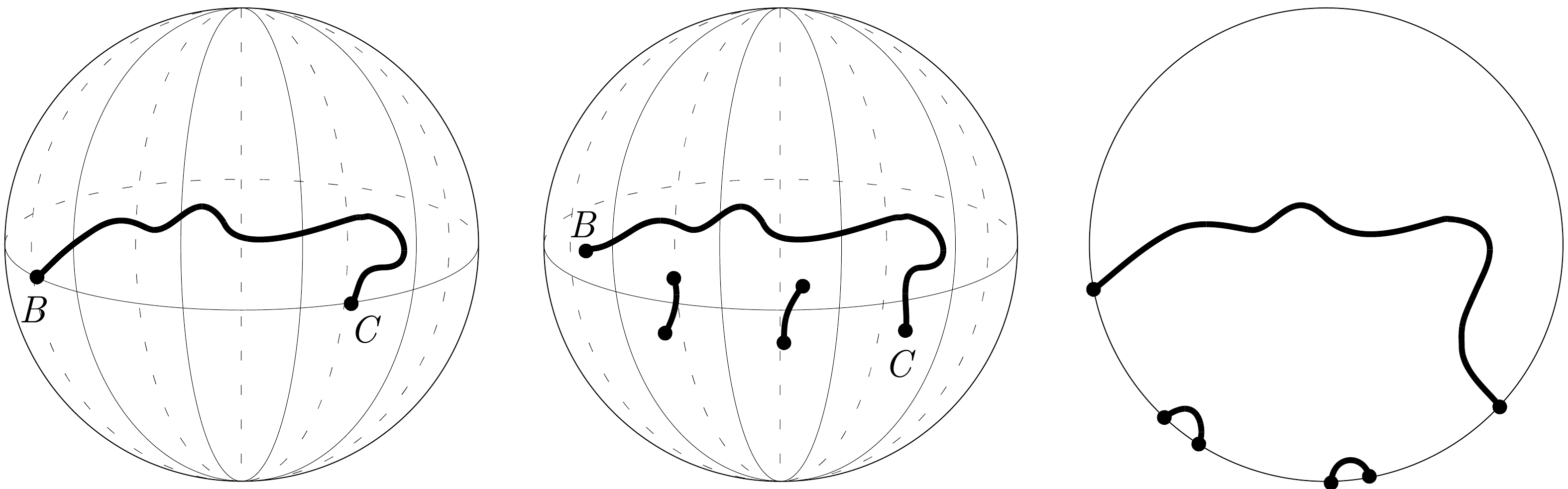}
\caption[legenda elenco figure]{Avoiding $\partial B^{2}_{0}$ in $L(9,1)$.}\label{complex}
\end{center}
\end{figure}

As a consequence, $L'$ is the disjoint union of closed curves in $\text{int} B^{3}$ and arcs properly embedded in $ B^{3}$ (i.e., only the boundary points belong onto $\partial B^{3}$).

Let $\mathbf{p} : B^{3} \smallsetminus \{ N,S \} \rightarrow B^{2}_{0}$ be the projection defined by $\mathbf{p}(x)=c(x) \cap B^{2}_{0}$, where $c(x)$ is the circle (possibly a line) through $N$, $x$ and $S$.
Take $L'$ and project it using $\mathbf{p}_{|L'}: L' \rightarrow B^{2}_{0}$. 
For $P \in \mathbf{p}(L')$, the set $\mathbf{p}_{|L'}^{-1}(P)$ may contain more than one point; in this case, we say that $P$ is a \emph{multiple point}. In particular, if it contains exactly two points, we say that $P$ is a \emph{double point}.
We can assume, by moving $L$ via a small isotopy, that the projection $\mathbf{p}_{|L'} : L' \rightarrow B^{2}_{0}$ of $L$ is \emph{regular}, namely:
\begin{enumerate}
\item[1)] the projection of $L'$ contains no cusps;
\item[2)] all auto-intersections of $\mathbf{p}(L')$ are transversal;
\item[3)] the set of multiple points is finite, and all of them are actually double points;
\item[4)] no double point is on $\partial B^{2}_{0}$.
\end{enumerate}

Now let $Q$ be a double point, consider $\mathbf{p}_{|L'}^{-1}(Q)=\{  P_{1}, P_{2} \}$ and suppose that $P_{1}$ is closer to $N$ than $P_{2}$. 
Let $U$ be a connected open neighborhood of $P_{2}$ in $L'$ such that $\mathbf{p}(\overline{U})$ contains no other double point and does not meet $\partial B_{0}^{2}$. We call $U$ \emph{underpass} relative to $Q$. Every connected component of the complement in $L'$ of all the underpasses (as well as its projection in $B^{2}_{0}$) is called \emph{overpass}. 

A \emph{diagram} of a link $L$ in $L(p,q)$ is a regular projection of $L'=F^{-1}(L)$ on the equatorial disk $B^{2}_{0}$, with specified overpasses and underpasses\footnote{As usual, the projections of the underpasses are not depicted in the diagram.} (see Figure~\ref{link3}). 

\begin{figure}[h!]                      
\begin{center}                         
\includegraphics[width=12cm]{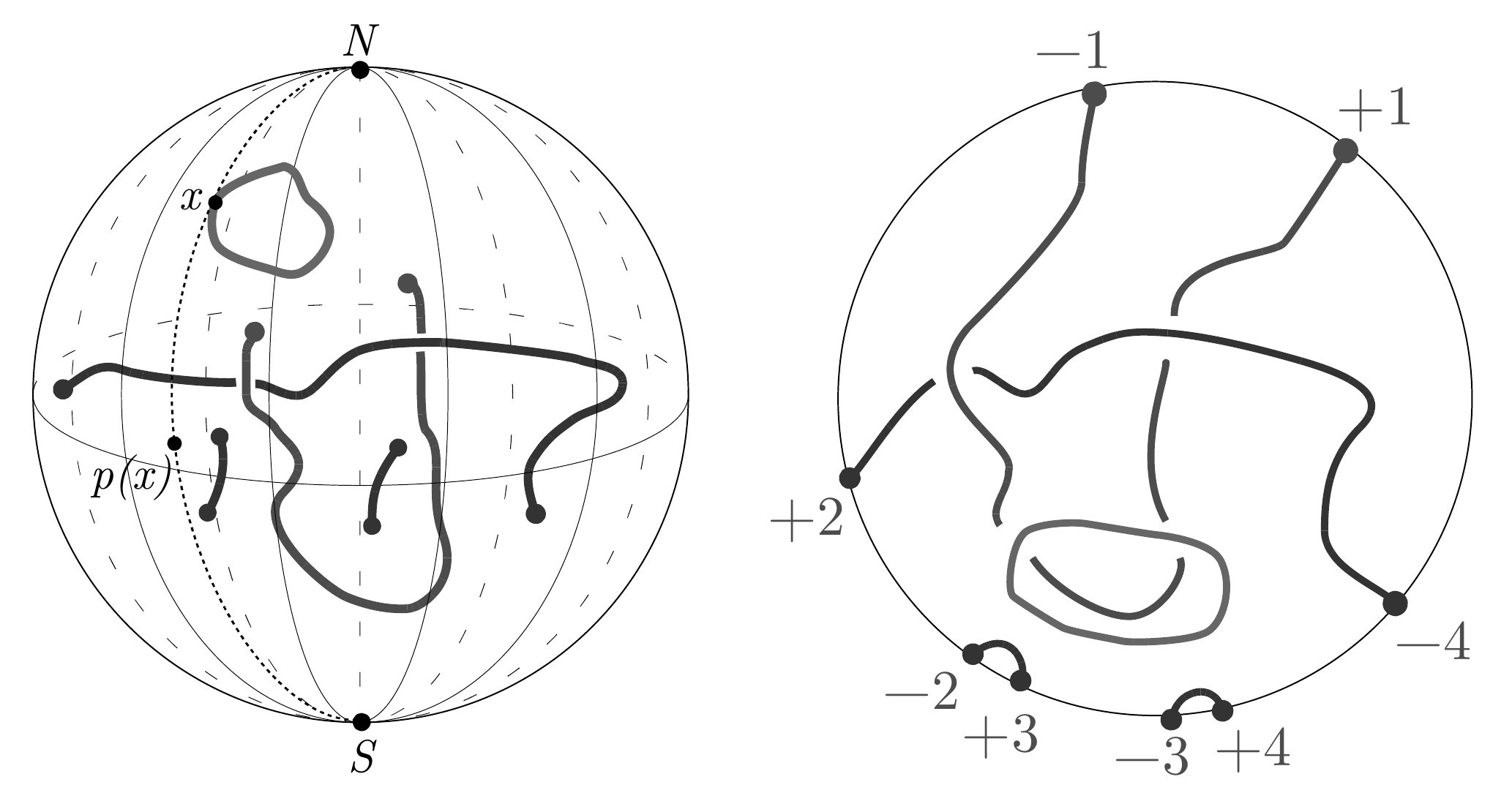}
\caption[legenda elenco figure]{A link in $L(9,1)$ and the corresponding diagram.}\label{link3}
\end{center}
\end{figure}

We assume that the equator is oriented counterclockwise if we look at it from $N$.
According to the orientation, label with $+1, \ldots, +t$ the endpoints of the overpasses belonging to the upper hemisphere, and with $-1, \ldots, -t$ the endpoints on the lower hemisphere, respecting the rule $+i \sim -i$. An example is shown in Figure~\ref{link3}.

Note that for the case $L(2,1) \cong \rp{3}$ we get exactly the diagram described in \cite{Dr}.

\end{section}

%============================================

\begin{section}{Generalized Reidemeister moves}

In this section we obtain a finite set of moves connecting two different diagrams of the same link.
The \emph{generalized Reidemeister moves} on a diagram of a link $L \subset L(p,q)$, are the moves $R_{1}, R_{2}, R_{3}, R_{4}, R_{5}, R_{6}$ and $R_{7}$ of Figure~\ref{vR1-R7}. Observe that, when $p=2$ the moves $R_{5}$ and $R_{6}$ are equal, and $R_{7}$ is a trivial move.

\begin{figure}[h!]                      
\begin{center}                         
\includegraphics[width=12.86cm]{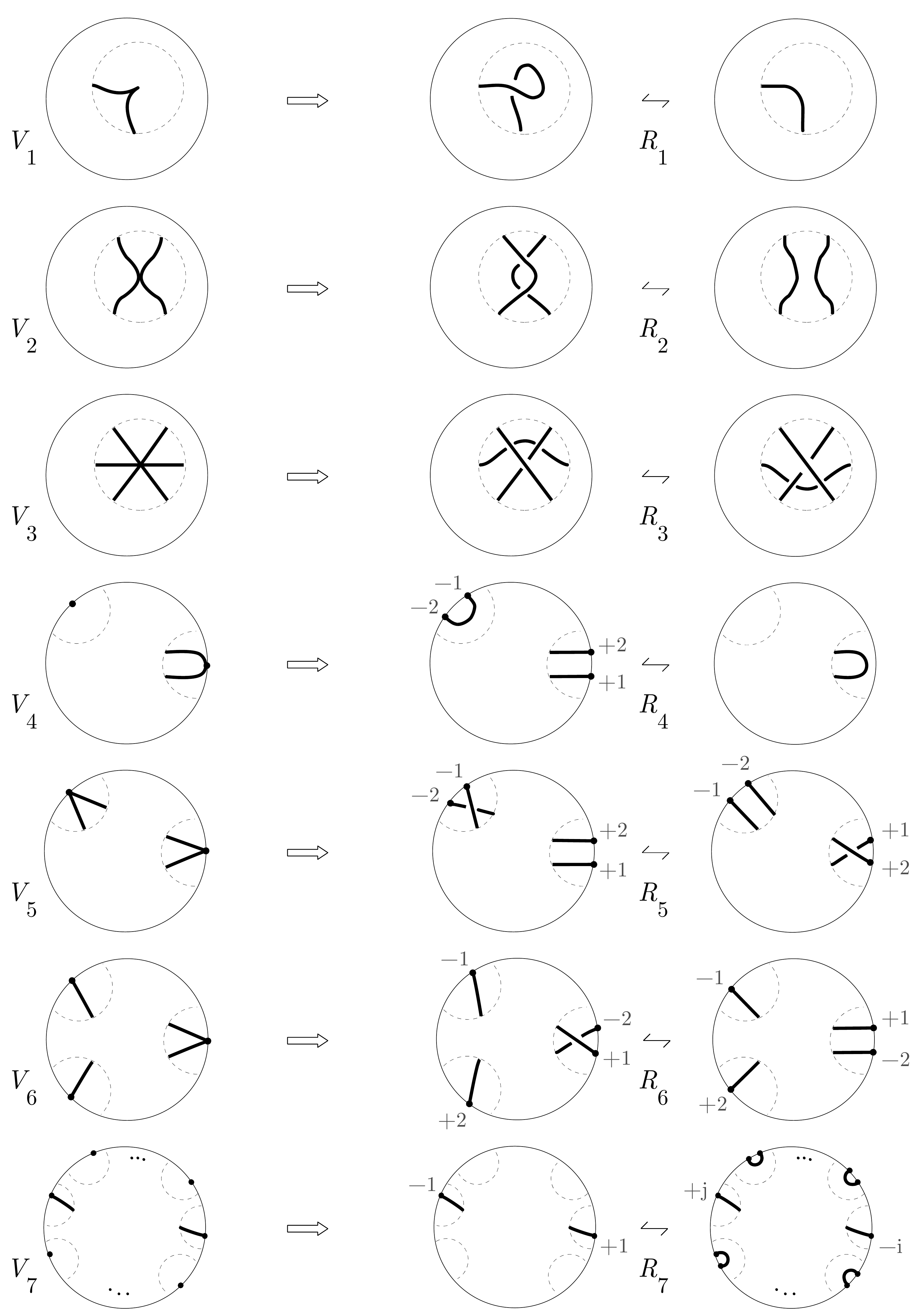}
\caption[legenda elenco figure]{Generalized Reidemeister moves.}\label{vR1-R7}
\end{center}
\end{figure}

\begin{teo}
Two links $L_{0}$ and $L_{1}$ in $L(p,q)$ are equivalent if and only if their diagrams can be joined by a finite sequence of generalized Reidemeister moves $R_{1}, \ldots, R_{7}$ and diagram isotopies, when $p>2$. If $p=2$, moves $R_{1}, \ldots, R_{5}$ are sufficient.
\end{teo}

\begin{proof}
It is easy to see that each Reidemeister move connects equivalent links, hence a finite sequence of Reidemeister moves and diagram isotopies does not change the equivalence class of the link.

On the other hand, if we have two equivalent links $L_{0}$ and $L_{1}$, then there exists an isotopy of the ambient space $H: L(p,q) \times [0,1] \rightarrow L(p,q)$ such that $h_{1}(L_{0})=L_{1}$. For each $t \in [0,1]$ we have a link $L_{t}=h_{t}(L_{0})$. 

The link $L_{t}$ may violate conditions i), ii), iii), iv) and its projection can violate the regularity conditions 1), 2), 3) and 4).

It is easy to see that the isotopy $H$ can be chosen in such a way that conditions i) and ii) are satisfied at any time. Moreover, using general position theory (see \cite{R} for details) we can assume that there are a finite number of forbidden configurations and that for each \mbox{$t \in [0,1]$}, only one of them may occur. 
The remaining conditions might be violated during the isotopy as depicted in the left part of Figure~\ref{vR1-R7}. More precisely, 
%\vspace{-6mm}
\begin{itemize}\itemsep-5pt
\item[--] conditions 1), 2) and 3) generate configurations $V_{1}$, $V_{2}$ and $V_{3}$;
\item[--] condition iii) generates $V_{4}$;
\item[--] condition 4) generates $V_{5}$ and $V_{6}$; the difference between the two configurations is that $V_{5}$ involves two arcs of $L'$ ending in the same hemisphere of $\partial B^{3}$, while $V_{6}$ involves arcs ending in different hemispheres;
\item[--] from condition iv) we have a family of configurations $V_{7,1}, \ldots, V_{7,p-1}$ (see Figure~\ref{v7f}); the difference between them is that $V_{7,1}$ has the endpoints of the projection identified directly by $g_{p,q}$, while $V_{7,k}$ has the endpoints identified by $g_{p,q}^{k}$, for $k=2, \ldots, p-1$.
\end{itemize}
\begin{figure}[h!]                      
\begin{center}                         
\includegraphics[width=14cm]{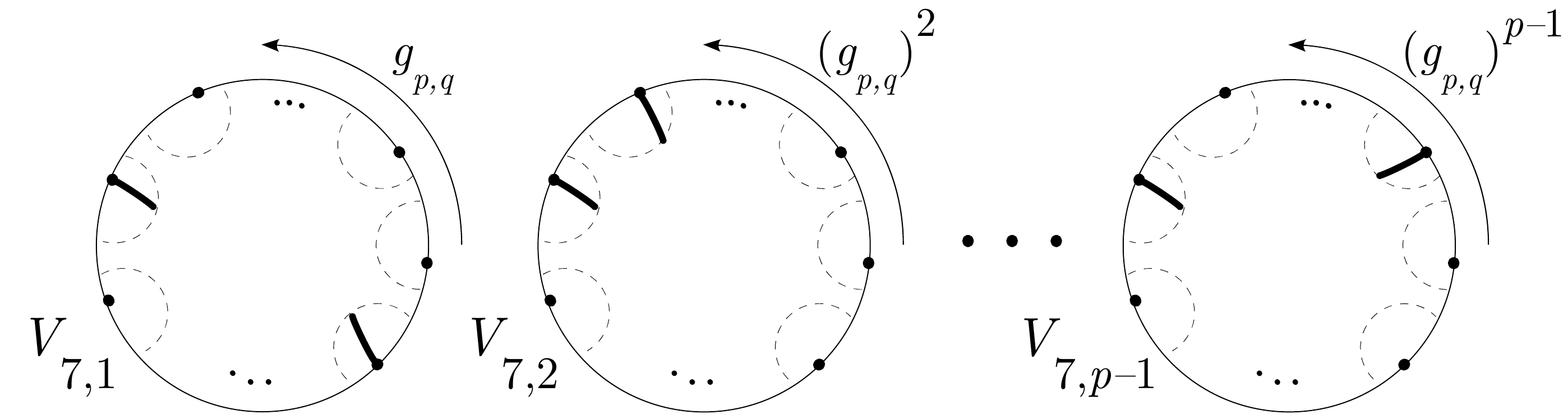}
\caption[legenda elenco figure]{Forbidden configurations $V_{7,1},V_{7,2}, \ldots, V_{7,p-1}$.}\label{v7f}
\end{center}
\end{figure}
\vspace{-1mm}
From each type of forbidden configuration a transformation of the diagram appears, i.e. a generalized Reidemeister move, as follows (see Figure~\ref{vR1-R7}): 
\vspace{-1mm}
\begin{itemize}\itemsep-5pt
\item[--] from $V_{1}$, $V_{2}$ and $V_{3}$ we obtain the usual Reidemeister moves $R_{1},R_{2}$ and $R_{3}$;
\item[--] from $V_{4}$ we obtain move $R_{4}$;
\item[--] from $V_{5}$, we obtain two different moves: $R_{5}$ if the overpasses endpoints belong to the same hemisphere, and $R_{6}$ otherwise;
\item[--] from $V_{7,1}, \ldots, V_{7,p-1}$ we obtain the moves $R_{7,1}, \ldots , R_{7,p-1}$.
\end{itemize}

%\vspace{-2mm}
\noindent Nevertheless the moves $R_{7,2}, \ldots, R_{7, p-1}$ can be seen as the composition of $R_{7}=R_{7,1}$, $R_{6}$, $R_{4}$ and $R_{1}$ moves. 
More precisely, the move $R_{7,k}$, with \mbox{$ k =2, \ldots , p-1$}, is obtained by the following sequence of moves: first we perform an $R_{7}$ move on the two overpasses corresponding to the points $+i$ and $-i$,  then we repeat $k-1$ times the three moves $R_{6}$-$R_{4}$-$R_{1}$ necessary to retract the small arc having the endpoints with the same sign (see an example in Figure~\ref{R7-R4}).

\begin{figure}[h!]\vspace{-10mm}
\begin{center}                         
\includegraphics[width=8.8cm]{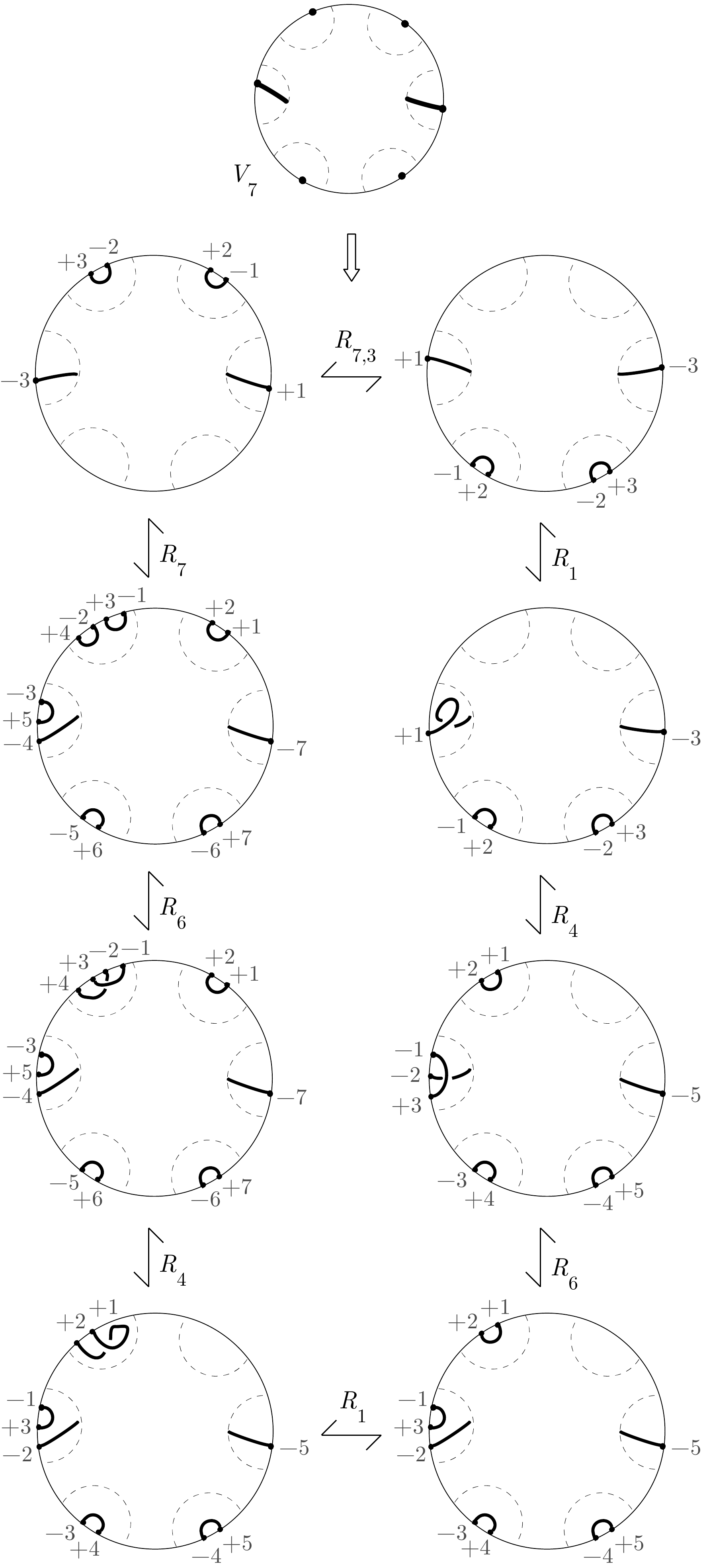}
\caption[legenda elenco figure]{How to decompose a move $R_{7,3}$.}\label{R7-R4}
\end{center}
\end{figure}

So we can drop out $R_{7,2}, \ldots, R_{7, p-1}$ from the set of moves and keep only $R_{7,1}=R_{7}$.
As a consequence, any pair of diagrams of two equivalent links can be joined by a finite sequence of generalized Reidemeister moves $R_{1}, \ldots, R_{7}$ and diagram isotopies. When $p=2$, it is easy to see that $R_{6}$ coincides with $R_{5}$, and $R_{7}$ is a trivial move; so in this case moves $R_{1}, \ldots, R_{5}$ are sufficient (see also \cite{Dr}).
\end{proof}

Diagram isotopies have to respect the identifications of boundary points of the link projection. Therefore, move $R_{6}$ is  possible only if there are no other arcs inside the small circles of the move $R_{6}$, as depicted in Figure \ref{vR1-R7}.
For example, Figure \ref{noR6} shows the case of a link in $L(3,1)$ where the $R_{6}$ move removing the crossing cannnot be performed.

\begin{figure}[h!]
\begin{center}                         
\includegraphics[width=5.5cm]{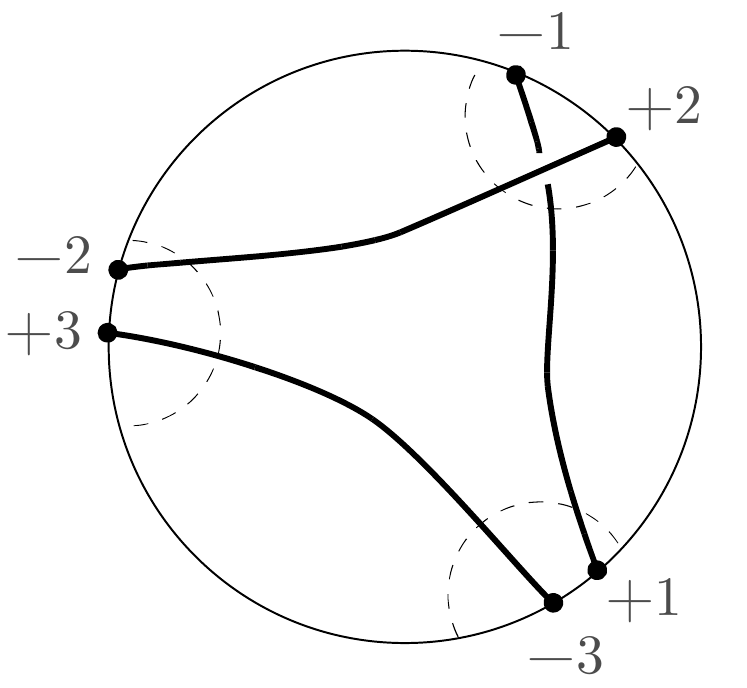}
\caption[legenda elenco figure]{A forbidden $R_6$ move.}\label{noR6}
\end{center}
\end{figure}

\end{section}

%=====================================================================

\begin{section}{Fundamental group}
In this section we obtain, directly from the diagram, a finite presentation for the fundamental group of the complement of links in $L(p,q)$.

Let $L$ be a link in $L(p,q)$, and consider a diagram of $L$. Fix an orientation for $L$, which induces an orientation on both $L'$ and $\mathbf{p}(L')$. Perform an $R_{1}$ move on each overpass of the diagram having both endpoints on the boundary of the disk; in this way every overpass has at most one boundary point. Then label the overpasses as follows: $A_{1}, \ldots, A_{t}$ are the ones ending in the upper hemisphere, namely in $+1, \ldots, +t$, while $A_{t+1}, \ldots, A_{2t}$ are the overpasses ending in $-1, \ldots, -t$. The remaining overpasses are labelled by $A_{2t+1},\ldots, A_{r}$.
For each $i=1 \ldots, t$, let $\epsilon_{i}=+1$ if, according to the link orientation, the overpass $A_{i}$ starts from the point $+i$; otherwise, if $A_{i}$ ends in the point $+i$, let $\epsilon_{i}=-1$.

\begin{figure}[h!]                      
\begin{center}                         
\includegraphics[width=12cm]{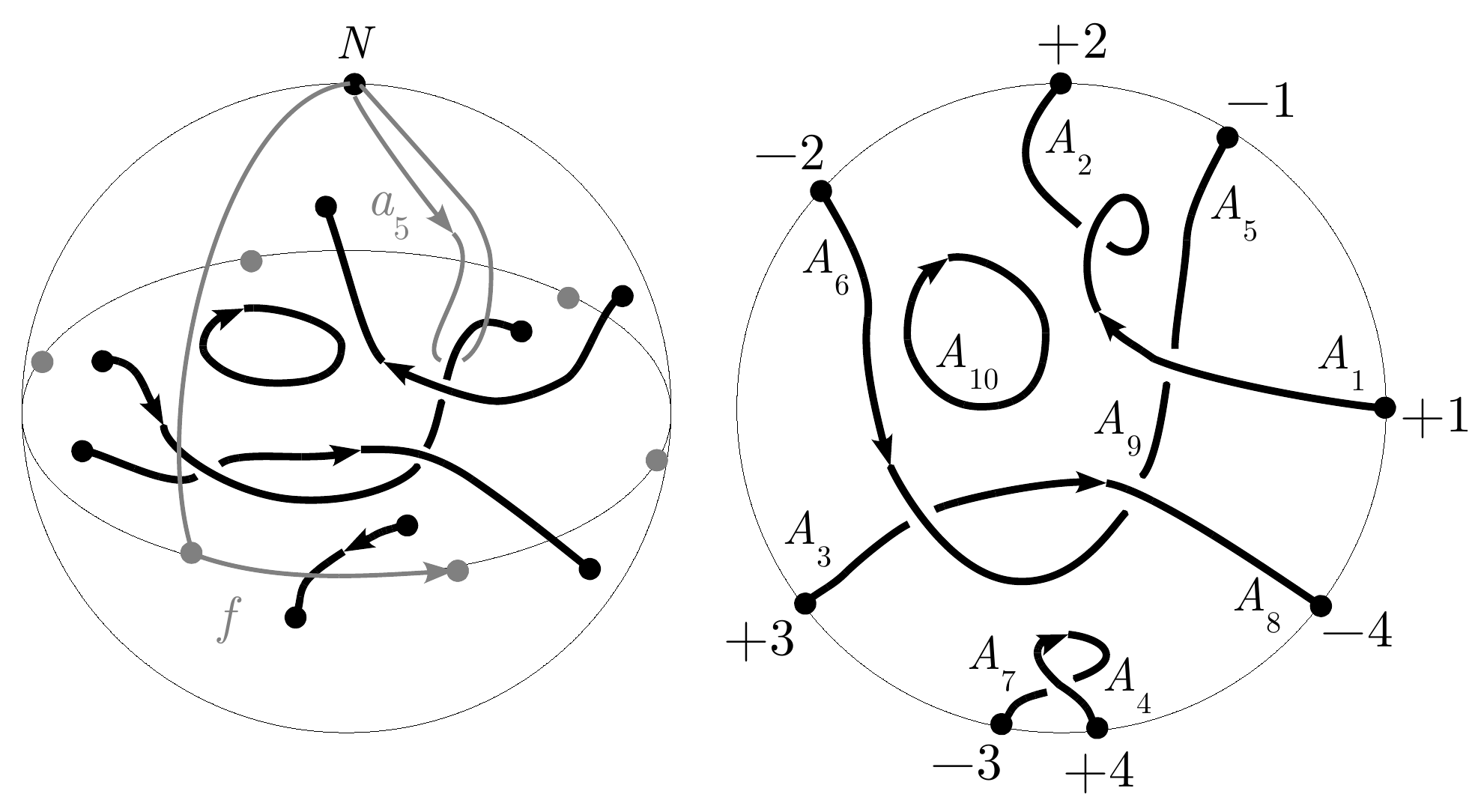}
\caption[legenda elenco figure]{Example of overpasses labelling for a link in $L(6,1)$.}\label{labelling}
\end{center}
\end{figure}

Associate to each overpass $A_{i}$ a generator $a_{i}$, which is a loop around the overpass as in the classical Wirtinger theorem, oriented following the left hand rule. Moreover let $f$ be the generator of the fundamental group of the lens space depicted in Figure~\ref{labelling}.
The relations are the following:
\begin{description}
\item[W:] $w_{1},\ldots , w_{s}$ are the classical Wirtinger relations for each crossing, that is to say $a_{i}a_{j}a_{i}^{-1}a_{k}^{-1}=1$ or $ a_{i}a_{j}^{-1}a_{i}^{-1}a_{k}=1$, according to Figure~\ref{wirtrel};
\begin{figure}[h!]                      
\begin{center}                         
\includegraphics[width=8cm]{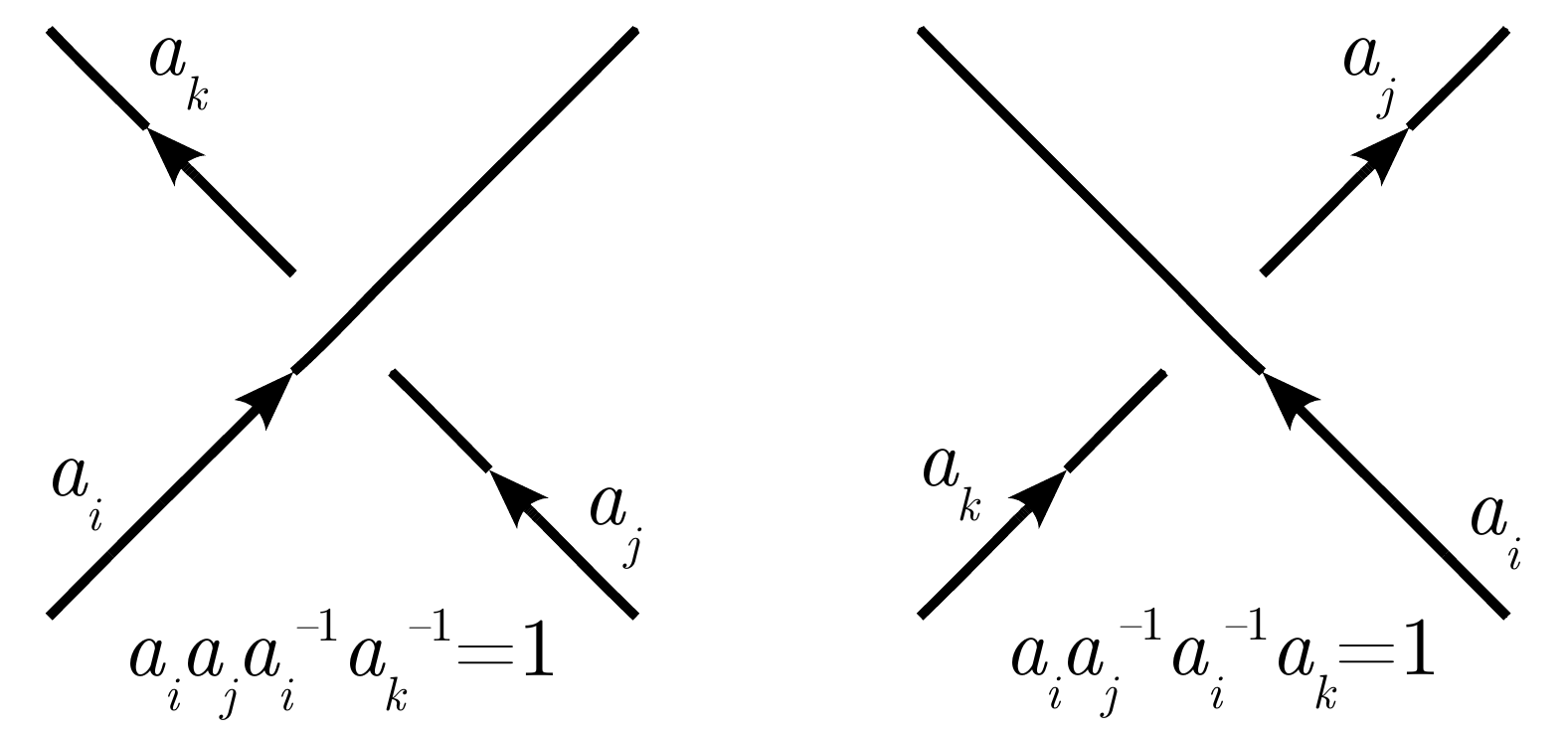}
\caption[legenda elenco figure]{Wirtinger relations.}\label{wirtrel}
\end{center}
\end{figure}
\vspace{-4mm}
\item[L:] $\ l$ is the lens relation $ a_{1}^{\epsilon_{1}} \cdots a_{t}^{\epsilon_{t}} = f^{p}$;
\item[M:] $m_{1}, \ldots, m_{t}$ are relations (of conjugation) between loops corresponding to overpasses with identified endpoints on the boundary. If $t=1$ the relation is $a_{2}^{\epsilon_{1}}=a_{1}^{-\epsilon_{1}}f^{q} a_{1}^{\epsilon_{1}}f^{-q} a_{1}^{\epsilon_{1}}$. Otherwise, consider the point $-i$ and, according to equator orientation, let $+j$ and $+j+1$ (mod $t$) be the type $+$ points aside of it.
We distinguish two cases:
\begin{itemize}
\item if $-i$ lies on the diagram between $-1$ and $+1$, then the relation $m_{i}$ is
\begin{multline*}
a_{t+i}^{\epsilon_{i}}= \big( \prod_{k=1}^{j} a_{k}^{\epsilon_{k}} \big)^{-1}  f^{q} \big( \prod_{k=1}^{i-1} a_{k}^{\epsilon_{k}} \big) \  a_{i}^{\epsilon_{i}} \
 \big( \prod_{k=1}^{i-1} a_{k}^{\epsilon_{k}} \big)^{-1} f^{-q} \big( \prod_{k=1}^{j} a_{k}^{\epsilon_{k}} \big); 
\end{multline*} 
\item otherwise, the relation $m_{i}$ is
\begin{multline*}
a_{t+i}^{\epsilon_{i}}= \big( \prod_{k=1}^{j} a_{k}^{\epsilon_{k}} \big)^{-1}  f^{q-p} \big( \prod_{k=1}^{i-1} a_{k}^{\epsilon_{k}} \big) \  a_{i}^{\epsilon_{i}} \
 \big( \prod_{k=1}^{i-1} a_{k}^{\epsilon_{k}} \big)^{-1} f^{p-q} \big( \prod_{k=1}^{j} a_{k}^{\epsilon_{k}} \big).
\end{multline*} 
\end{itemize}
\end{description}

\begin{teo} \label{lpqio}
Let $\ast=F(N)$, then the group of the link $L \subset L(p,q)$ is:
$$
 \pi_{1} (L(p,q) \smallsetminus L, \ast)=  \langle a_{1}, \ldots ,a_{r}, f \ 
 | \  w_{1},\ldots , w_{s}, %v_{1}, \ldots , v_{u},
 l, m_{1},\ldots,m_{t} \rangle .
$$
 \end{teo}

\begin{proof}
Suppose that $L'=F^{-1}(L)$ is such that $\mathbf{p}_{|L'}: L' \rightarrow B_{0}^{2}$ is a regular projection. Consider a sphere $\s{2}_{\varepsilon}$ of radius $1-\varepsilon$, with $0< \varepsilon < 1$; this sphere splits the $3$-ball $B^{3}$ into two parts: call $B_{\varepsilon}^{3}$ the internal one and  $E_{\varepsilon}$ the external one.
Choose $\varepsilon$ small enough such that all the underpasses belong into $\text{int} (B_{\varepsilon}^{3})$. Let $N_{\varepsilon}$ be the north pole of $B_{\varepsilon}^{3}$, and consider $ \tilde{\mathbf{S}}^{2}_{\varepsilon} = \s{2}_{\varepsilon} \cup\overline{N N_{\varepsilon}} $.

In order to compute $\pi_{1}(L(p,q) \smallsetminus L,\ast)$, we apply Seifert-Van Kampen theorem with decomposition  $(L(p,q) \smallsetminus L)=(F(\tilde B^{3}_{\varepsilon}) \smallsetminus  L) \cup (F(E_{\varepsilon}) \smallsetminus L)$.

The fundamental group of $F(\tilde B^{3}_{\varepsilon}) \smallsetminus  L$ can be obtained as in the classical Wirtinger Theorem:
$$ \pi_{1}(F(\tilde B^{3}_{\varepsilon}) \smallsetminus  L, \ast)= \langle a_{1}, \ldots , a_{r} \ | \ w_{1}, \ldots ,w_{s} %, v_{1}, \ldots, v_{u}
  \rangle.$$
  
For $F(E_{\varepsilon}) \smallsetminus L$, we proceed in the following way: first of all observe that we can retract $F(E_{\varepsilon})  \smallsetminus L $ to $E \smallsetminus L$, where $E$ is $\partial B^{3} / \sim$.
According to the orientation, fix a point $T_{1}$ in $\partial B^{2}_{0}$ just before $+1$ and such that its equivalent points $T_{2}, \ldots, T_{p}$ (via $\sim$) do not belong to $\mathbf{p}(L')$. 
\begin{figure}[h!]                  
\begin{center}                            
 \includegraphics[width=10cm]{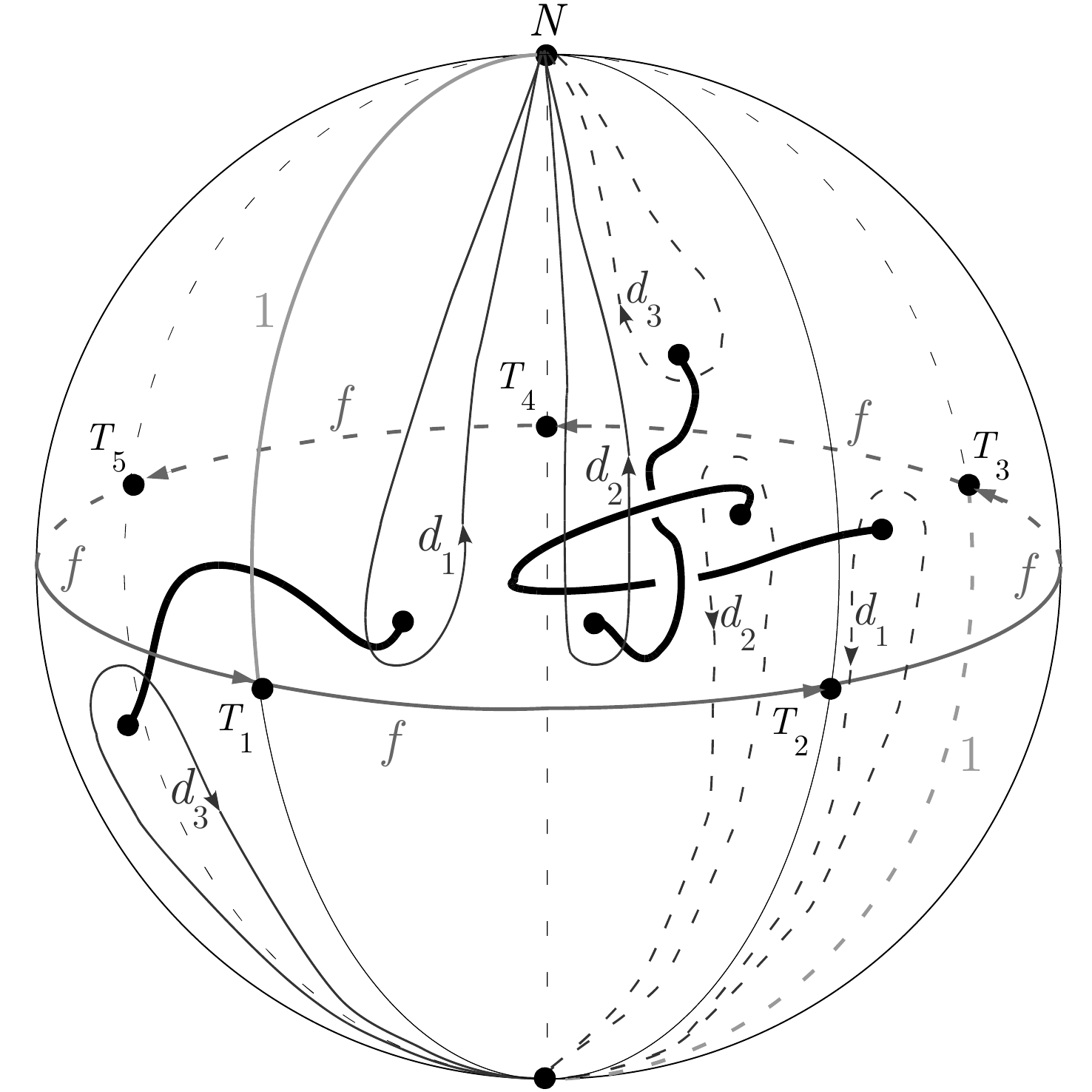}
\caption[legenda elenco figure]{Boundary complex for a knot in $L(5,2)$.}\label{LCW} 
\end{center}
\end{figure}
Following the example of Figure~\ref{LCW}, the $2$-complex $E$ is a CW-complex composed by: two $0$-cells $N=S$ and $T_{1}=T_{2}=\ldots =T_{p}$, two $1$-cells $\wideparen{NT_{1}}$ (chosen as a maximal tree in the 1-skeleton) and $\wideparen{T_{1}T_{2}}$ (corresponding to $f$), and one $2$-cell, that is the upper hemisphere.
In order to obtain $\pi _{1} (E \smallsetminus L, \ast)$, we need to add the loops $d_{1}, \ldots, d_{t}$ around the points of $L$.
The relation given by the $2$-simplex is $d_{1} \cdots d_{t}=f^{p}$.
Hence the fundamental group of $E \smallsetminus L$ is:
\begin{equation}\label{DFref}
 \pi_{1}(E \smallsetminus L  , \ast)= \langle d_{1}, \ldots , d_{t}, f \ | \ d_{1} \cdots d_{t}=f^{p} \, \rangle. 
\end{equation}

\begin{figure}[h!]                      
\begin{center}                         
\includegraphics[width=10cm]{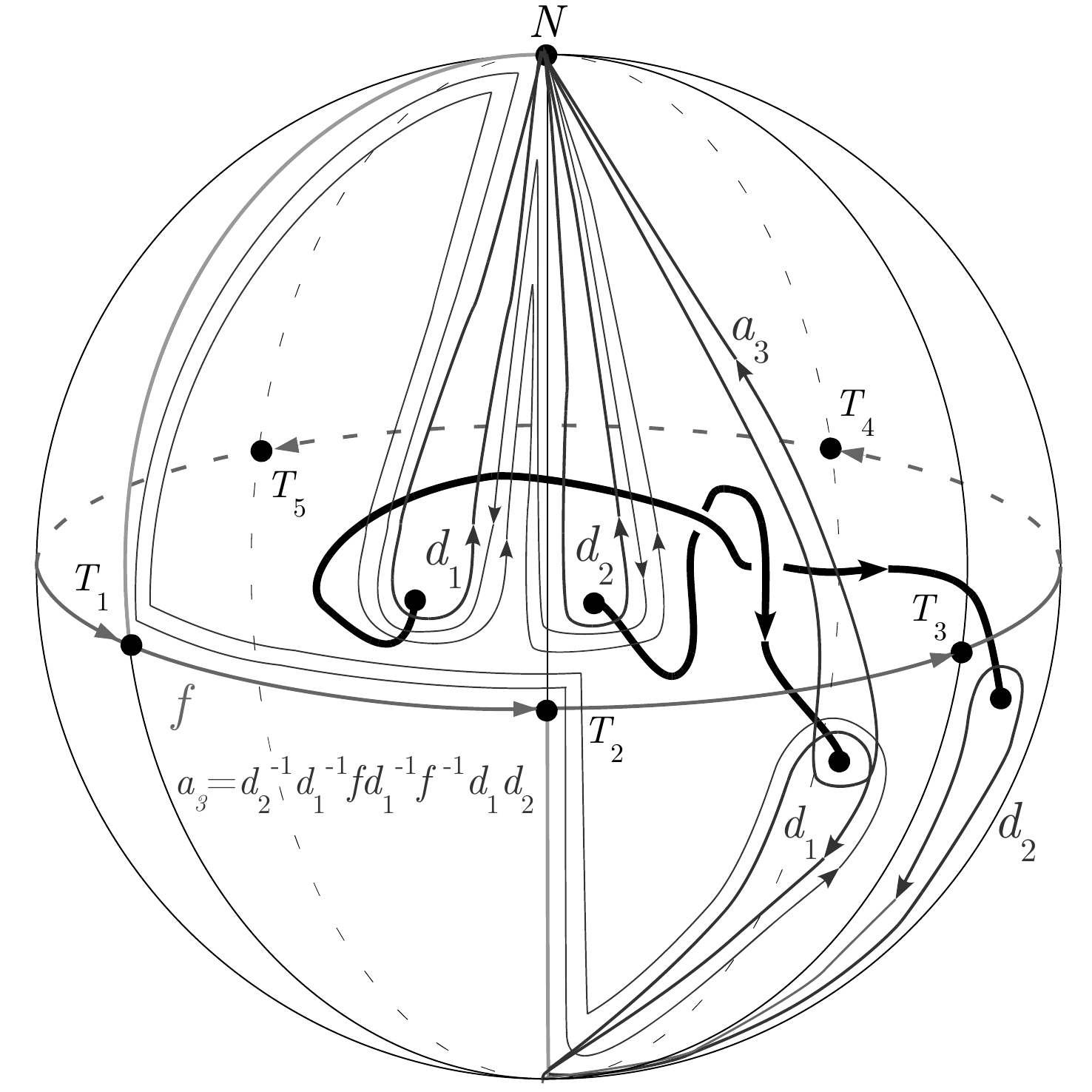}
\caption[legenda elenco figure]{Example of relation for a link in $L(5,1)$.}\label{Irel}
\end{center}
\end{figure}

Finally, the fundamental group of $F(\tilde{\bold{S}}^{2}_{\varepsilon}) \smallsetminus L = (F(\tilde B^{3}_{\varepsilon}) \smallsetminus L) \cap (F(E_{\varepsilon}) \smallsetminus L)$ is generated by $a_{1}, \ldots, a_{2t}$.
By Seifert-Van Kampen theorem, we identify each $a_{1}, \ldots, a_{t}$ with the corresponding generator $d_{1}, \ldots, d_{t}$, according to orientation:
$a_{i}^{\epsilon_{i}}=d_{i}$.
Furthermore we need to identify $a_{t+1}, \ldots a_{2t}$ with suitable loops in the CW-complex, distinguishing two cases:

\begin{description}
\item[-] if $-i$ lies on the diagram between $-1$ and $+1$, then we obtain the following relation (see Figure~\ref{Irel} for an example)
$$
\quad a_{t+i}^{\epsilon_{i}}= \big( \prod_{k=1}^{j} d_{k} \big)^{-1}  f^{q} \big( \prod_{k=1}^{i-1} d_{k} \big) \  d_{i} \
 \big( \prod_{k=1}^{i-1} d_{k} \big)^{-1} f^{-q} \big( \prod_{k=1}^{j} d_{k} \big); 
$$
\item[-] otherwise, the relation is
$$
a_{t+i}^{\epsilon_{i}}= \big( \prod_{k=1}^{j} d_{k} \big)^{-1}  f^{q-p} \big( \prod_{k=1}^{i-1} d_{k} \big) \  d_{i} \
 \big( \prod_{k=1}^{i-1} d_{k} \big)^{-1} f^{p-q} \big( \prod_{k=1}^{j} d_{k} \big). 
$$
\end{description}
At last we remove $d_{1}, \ldots, d_{t}$ from the group presentation, obtaining:
$$
 \pi_{1} (L(p,q) \smallsetminus L, \ast)=  \langle a_{1}, \ldots ,a_{r}, f \ 
 | \  w_{1},\ldots , w_{s}, %v_{1}, \ldots , v_{u}, 
 l, m_{1},\ldots,m_{t}  \rangle . \qedhere
$$
\end{proof}

In the special case of $L(2,1)=\rp{3}$, the presentation is equivalent (via Tietze transformations) to the one given in \cite{HL}.

\begin{oss}
\label{ridotto} 
If the link diagram does not contain overpasses which are circles (we can avoid this case by using suitable $R_{1}$ moves), then the presentation of Theorem~\ref{lpqio} is balanced (i.e., the number of generators equals the number of relations). Indeed, it is enough to think at each intersection between the diagram and the boundary disk as a fake crossing. Moreover, the product of the Wirtinger relators represents a loop that is trivial in $\pi_{1}(E \smallsetminus   L , \ast)$, so anyone of the Wirtinger relations can be deduced from the others, obtaining a presentation of deficiency one.
\end{oss}

\end{section}

%======================================================================

\begin{section}{First homology group}

In this section we show how to determine, directly from the diagram, the first homology group of links in $L(p,q)$, which is useful for the computation of twisted Alexander polynomials.

Consider a diagram of an oriented knot $K \subset L(p,q)$ and let $\epsilon_i$ be as defined in the previous section. If $n_{1}=| \{ \epsilon_{i}  \ | \ \epsilon_{i} =+1 ,\ i= 1, \ldots, t\} |$ and $n_{2}=| \{ \epsilon_{i}  \ | \ \epsilon_{i} =-1, \  i= 1, \ldots, t\}|$, define $\delta_{K}=q(n_{2}-~n_{1}) \mod p$.

\begin{lem}\label{lemhom}
If $K \subset L(p,q)$ is an oriented knot and $[K]$ is the homology class of $K$ in $H_{1}(L(p,q))$, then $[K]=\delta_{K}$.
\end{lem}

\begin{proof} Let $f$ be the generator of $H_{1}(L(p,q)) =\mathbb{Z}_{p}$, as depicted in Figure \ref{hom}. 
Let $K\cap(\partial B^3/\sim )=\{P_1,\ldots,P_t\}$. For $i=1,\ldots,t$, consider the identification class $[P_i]_{\sim}=\{P'_i,P''_i\}$, with $P'_i\in E_+$ and $P''_i\in E_-$. Denote with $\gamma_i$ the path (actually a loop in $L(p,q)$) connecting $P'_i$ with $P''_i$ as in Figure~\ref{hom}, oriented as depicted if $\epsilon_i=+1$ and in the opposite direction if $\epsilon_i=-1$. Of course its homology class is $[\gamma_i]=\epsilon_i q$. The loop $K'=K\cup\gamma_1\cup\cdots\cup\gamma_t$ is homologically trivial, so we have: $0=[K']=[K]+\sum_{i=1}^t[\gamma_i]=[K]+(n_1-n_2)q$, and therefore $[K]=\delta_K$.

\begin{figure}[h!]                      
\begin{center}                         
\includegraphics[width=6.8cm]{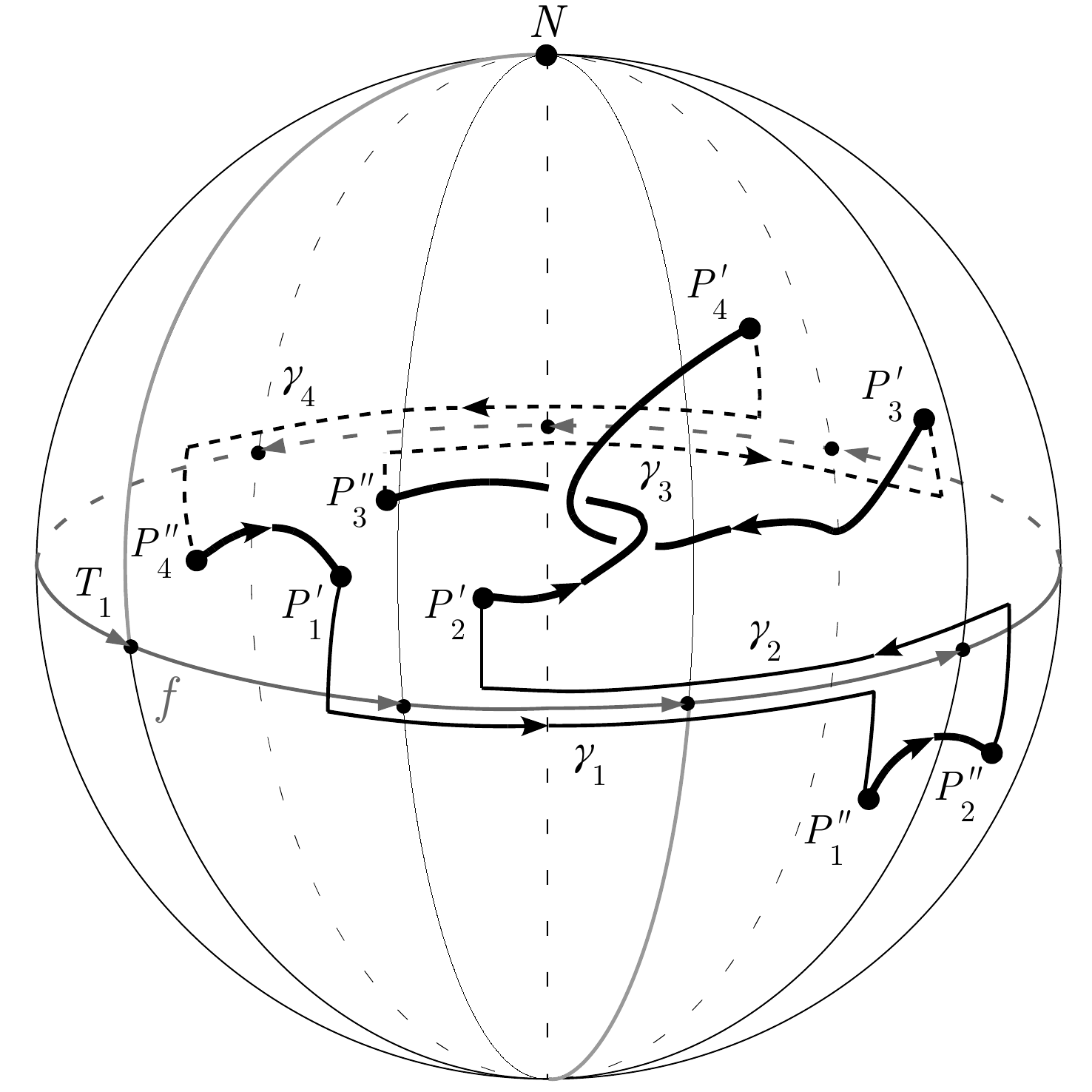}
\caption[legenda elenco figure]{Equatorial arcs for a knot in $L(7,2)$.}\label{hom}
\end{center}
\end{figure}
\end{proof}

\vbox{
\begin{cor}\label{homology}
Let $L$ be a link in $L(p,q)$, with components $L_{1}, \ldots L_{\nu}$.
For each $j=1, \ldots, \nu$, let $\delta_{j} =[L_{j}] \in \mathbb{Z}_{p}=H_{1}(L(p,q))$.
Then $$ H_{1}( L(p,q) \smallsetminus L)  \cong \mathbb{Z}^{\nu} \oplus \mathbb{Z}_{d},$$ where $d=\gcd( \delta_{1}, \ldots , \delta_{\nu}, p)$.
\end{cor}
}

\begin{proof} 
We abelianize the fundamental group presentation given in Section~4. Relations of type W and M imply that generators corresponding to the same link component are homologous. So $ H_{1}( L(p,q) \smallsetminus L) $ is generated by $g_{1}, \ldots , g_{\nu}$, which are generators corresponding to the link components, and $f$. Relation L becomes: $pf-( \tilde \delta_{1} g_{1}+ \ldots + {\tilde \delta_{\nu}}  g_{\nu}) =0 $, with $\tilde \delta_{j}= \sum_{A_{h} \subset L_{j}} \epsilon_{h}$, where $L_j$ is the $j$-th component of $L$. Therefore $ H_{1}( L(p,q) \smallsetminus L)  \cong \mathbb{Z}^{\nu} \oplus \mathbb{Z}_{d}$, where $d=\gcd( \tilde \delta_{1}, \ldots , \tilde \delta_{\nu}, p)$. Since $\gcd(p,q)=1$ and, by Lemma~\ref{lemhom}, $\delta_{j}=-q\tilde\delta_{j}$,
we obtain $ d= \gcd(\tilde \delta_{1}, \ldots,\tilde  \delta_{\nu}, p)=\gcd(\delta_{1}, \ldots, \delta_{\nu}, p). $
\end{proof}

\end{section}

%===================================================================

\begin{section}{Twisted Alexander polynomials}

In this section we analyze the twisted Alexander polynomials of links in lens spaces and their relationship with Reidemeister torsion. Start by recalling the definition of twisted Alexander polynomials (for further references see \cite{T}). 
Given  a finitely generated  group $\pi$, denote with \hbox{$H=\pi/\pi'$} its abelianization and let $G=H/\textup{Tors}(H)$. Take a presentation  \hbox{$\pi=\langle x_1,\ldots,x_m\mid r_1\ldots,r_n\rangle$} and consider the Alexander-Fox matrix $A$ associated to the presentation, that is $A_{ij}=\textup{pr}(\frac{\partial r_i}{\partial x_j})$, where $\textup{pr}$ is the natural projection $\mathbb Z[F(x_1,\ldots,x_m)]\to\mathbb Z[\pi]\to\mathbb Z[H]$ and $\frac{\partial r_i}{\partial x_j}$ is the Fox derivative of $r_i$. Moreover let $E(\pi)$ be the  first elementary ideal of $\pi$, which is the ideal of $\mathbb Z[H]$ generated by the \mbox{$(m-1)$-minors} of $A$.   For each  homomorphism $\sigma:\textup{Tors}(H)\to \mathbb C^*=\mathbb C \smallsetminus \{ 0 \}$ we can define a twisted Alexander polynomial $\Delta^{\sigma}(\pi)$ of $\pi$ as follows: fix a splitting $H=\textup{Tors}(H)\times G$ and consider the ring homomorphism that we still denote with  $\sigma: \mathbb Z[H]\to \mathbb C[G]$ sending $(f,g)$, with $f\in\textup{Tors}(H)$ and $g\in G$, to $\sigma(f)g$, where \mbox{$\sigma(f)\in \mathbb C^*$}. The ring $\mathbb C[G]$ is a unique factorization domain and we set $\Delta^{\sigma}(\pi)=\gcd(\sigma(E(\pi))$. This is an element of $\mathbb C[G]$ defined up to multiplication by elements of $G$ and non-zero complex numbers. If $\Delta(\pi)$ denote the classic Alexander polynomial we have $\Delta^1(\pi)=\alpha \Delta(\pi)$, with $\alpha \in \mathbb C^*$.

If $L\subset L(p,q)$ is a link in a lens space then the $\sigma$-\textit{twisted Alexander polynomial} of $L$ is  $\Delta^{\sigma}_L=\Delta^{\sigma}(\pi_1(L(p,q) \smallsetminus L))$.  Since in this case $\textup{Tors}(H)=\mathbb Z_d$ then $\sigma(\textup{Tors}(H))$ is contained in the cyclic group generated by $\zeta$, where  $\zeta$ is a $d$-th primitive root of the unity. When $\mathbb Z[\zeta]$ is a principal ideal domain, in order to define $\Delta^{\sigma}_L$ we can consider the restriction $\sigma:\mathbb Z[H]\to \mathbb Z[\zeta][G]$. Note that $\Delta^{\sigma}_L\in\mathbb Z[\zeta][G]$ is defined up to multiplication by $\zeta^hg$, with $g\in G$.  In this setting we recall the following theorem.

\begin{prop}{\upshape\cite{MM}}
If $\zeta$ is a $d$-th primitive root of unity, then  the ring $\mathbb{Z}[\zeta]$ is a principal ideal domain if and only if $d\cong 2\ \mod 4$ or $d$ is one of the following 30 integers: 1, 3, 4, 5, 7, 8, 9, 11, 12, 13, 15, 16, 17, 19, 20, 21, 24, 25, 27, 28, 32, 33, 35, 36, 40, 44, 45, 48, 60, 84.
\end{prop}

A link is called \textit{local} if it is contained in a ball embedded in $L(p,q)$. For local links the following properties hold.

\begin{prop}
Let $L$ be a  local link in $L(p,q)$. Then $\Delta^{\sigma}_L=0$ if $\sigma\ne 1$, and $\Delta_L= p \cdot \Delta_{\bar{L}}$ otherwise, where $\bar L$ is the link L considered as a link in $\s3$.
\end{prop}
\begin{proof}
The fundamental group of $L$ can be presented with the relations of Wirtinger type and the lens relation $f^p=1$ only. Therefore the column in the Alexander-Fox matrix $A$ corresponding to the Fox derivative of the lens relation is everywhere zero except for the entry corresponding to the $f$-derivative, which is \mbox{$1+f+f^2+\cdots+f^{p-1}$}.  Moreover, the cofactor of this non-zero entry is equal to the Alexander-Fox matrix of $\bar{L}$.  So the statement follows by observing that in the case of $\Delta_L$, the generator $f$ is sent to 1, while if $\sigma\neq 1$, the generator $f$ is sent in a $k$-th root of the unity, where $k$ divides $p$, and so $\sigma(1+f+f^2+\cdots+f^{p-1})=0$.
\end{proof}

As a consequence a knot with a non trivial twisted Alexander polynomial cannot be local.  

Figure~\ref{tavola1} shows the twisted Alexander polynomials of a local trefoil knot in $L(4,1)$ and proves that twisted Alexander polynomial may distinguish knots with the same Alexander polynomial.

\begin{figure}[htb]
\begin{minipage}[c]{0.21\linewidth}
\centering
\includegraphics[width=3.5cm]{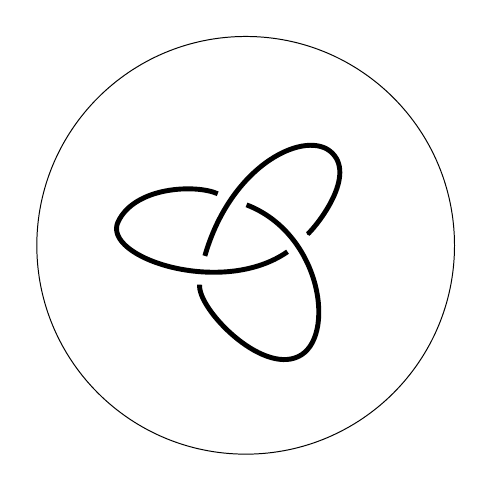} 
\end{minipage}
\begin{minipage}[c]{0.24\linewidth}
\centering
\begin{tabular}{c}
$\Delta^{1}_{T}=4(t^2-t+1)$ \\
 $\Delta^{-1}_{T}=0$ \\
  $\Delta^{i}_{T}=0$ \\
   $\Delta^{-i}_{T}=0$ 
\end{tabular}
\end{minipage}
\begin{minipage}[c]{0.09\linewidth}
\end{minipage}
\begin{minipage}[c]{0.23\linewidth}\centering
\includegraphics[width=3.5cm]{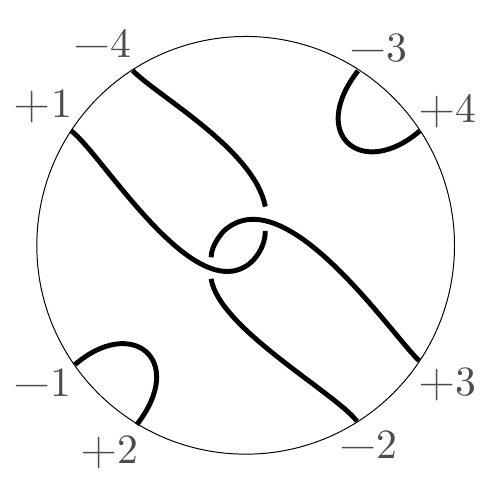}
\end{minipage}
\begin{minipage}[c]{0.23\linewidth}\centering
\begin{tabular}{c}
$\Delta^{1}_{K}=4(t^2-t+1)$ \\
 $\Delta^{-1}_{K}=0$ \\
  $\Delta^{i}_{K}=2(t-1)$ \\
   $\Delta^{-i}_{K}=2(t-1)$ 
\end{tabular}
\end{minipage}
\caption[legenda elenco figure]{Twisted Alexander polynomials for two knots in $L(4,1)$.}\label{tavola1}
\end{figure}
Let $L=L_1\sharp L_2$, where $\sharp$ denote  the connected sum and $L_2$ is a local link.  The decomposition $(L(p,q),L)=(L(p,q),L_1)\sharp (\s3 ,L_2)$ induces  monomorphisms $j_1:H_1(L(p,q) \smallsetminus L_1)\to H_1(L(p,q)\smallsetminus L)$ and $j_2:H_1(\s3\smallsetminus L_2)\to H_1(L(p,q)\smallsetminus L)$. Given $\sigma:\mathbb Z [H_1(L(p,q) \smallsetminus L)]\to \mathbb C[G]$  induced by 
\hbox{$\sigma\in\hom(\textup{Tors}(H_1(L(p,q) \smallsetminus L)),\mathbb C^*)$,}  denote with  $\sigma_1$ and $\sigma_2$ its restrictions to $\mathbb Z[j_1(H_1(L(p,q)\smallsetminus L_1))]$  and $\mathbb Z[j_2(H_1(\s3\smallsetminus L_2))]$ respectively. We have the following result.

\begin{prop}

Let $L=L_1\sharp L_2\subset L(p,q)$, where  $L_2$ is local link.  
With the above notations we have $\Delta_L^{\sigma}=\Delta_{L_1}^{\sigma_1}\cdot  \Delta_{L_2}^{\sigma_2}$.

\end{prop}
\begin{proof}
Since $(L(p,q),L)=(L(p,q),L_1)\sharp (\s3,L_2)$, by Van Kampen theorem we get  \hbox{$\pi_1(L(p,q) \smallsetminus L)=\langle a_1,\ldots ,a_n,b_1,\ldots ,b_m\mid r_1,\ldots ,r_{n-1},s_1,\ldots,s_{m-1},a_1=b_1\rangle$,} where $\pi_1(L(p,q)\setminus L_1,*)=\langle a_1,\ldots ,a_n\mid r_1,\ldots ,r_{n-1}\rangle$ and  $\pi_1(\s3\setminus L_2, *)=\langle b_1,\ldots ,b_m\mid s_1,\ldots ,s_{m-1}\rangle$. So  the Alexander-Fox matrix  of $L$ is 
$$A_L=\left(\begin{array}{cccccc}\ &A_{L_1}&\ &\  &0 &\ \\\ & 0 \ &\ & \ &A_{L_2}&\ \\-1\ 0&\cdots& 0&1\ 0&\cdots&\ 0\end{array}\right),$$
where $A_{L_i}$ is the Alexander-Fox matrix of $L_i$, for $i=1,2$. If $d_k(A)$ denotes the greatest common division of all $k$-minors of a matrix $A$, then a simple computation shows that $d_{m+n-1}(A_L)=d_{n-1}(A_{L_1})\cdot d_{m-1}(A_{L_2})$. Therefore it is easy to see that $\Delta_L^{\sigma}=\Delta_{L_1}^{\sigma_1}\cdot  \Delta_{L_2}^{\sigma_2}$.
\end{proof}

In Figure~\ref{tavola2} we compute the twisted Alexander polynomials of the connected sum of a local trefoil knot $\overline{T}$ with the three knots \hbox{$K_{0},K_{1},K_{2}\subset L(4,1)$} depicted in the left part of the figure, respectively. Note that for the case of $K_{2} \sharp \overline{T} $, the map $\sigma_{2}$, that is the restriction of $\sigma$ to $\mathbb Z[j_2(H_{1}(\s3\smallsetminus \overline{T}))]$,  sends the generator $g \in \mathbb{Z} [H_{1}(S^{3} \smallsetminus \overline{T})] $  in $t^{2} \in \mathbb{Z} [H_{1}(L(p,q) \smallsetminus K_{2} \sharp \overline{T})] $ (resp. in $-t^2$) if $\sigma=1$ (resp. if $\sigma=-1$), instead of $t$ as it does for the classical Alexander polynomial.

\begin{figure}[htb]

\begin{minipage}[c]{0.21\linewidth}\centering
\includegraphics[width=3.5cm]{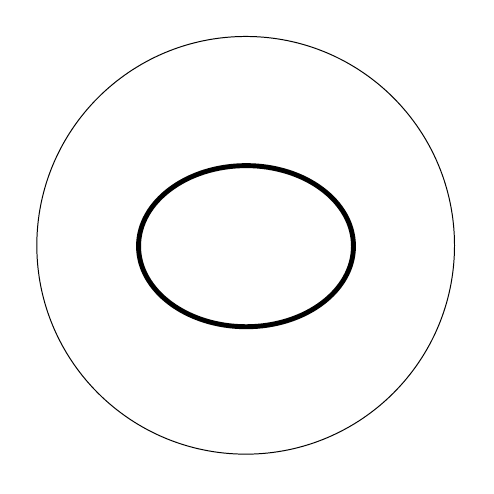}
\end{minipage}
\begin{minipage}[c]{0.21\linewidth}\centering
\begin{tabular}{c}
$\Delta^{1}_{K_{0}}=4$ \\
 $\Delta^{-1}_{K_{0}}=0$ \\ 
 $\Delta^{i}_{K_{0}}=0$ \\
  $\Delta^{-i}_{K_{0}}=0$ 
\end{tabular}
\end{minipage}
\begin{minipage}[c]{0.06\linewidth}
\end{minipage}
\begin{minipage}[c]{0.22\linewidth}
\centering
\includegraphics[width=3.5cm]{l41nodotrifoglio0omol.pdf}
\end{minipage}
\begin{minipage}[c]{0.23\linewidth}
\centering
\begin{tabular}{c}
$\Delta^{1}_{K_{0} \sharp \overline{T}}=4(t^2-t+1)$ \\
 $\Delta^{-1}_{K_{0} \sharp \overline{T}}=0$ \\
  $\Delta^{i}_{K_{0} \sharp \overline{T}}=0$ \\
   $\Delta^{-i}_{K_{0} \sharp \overline{T}}=0$ 
\end{tabular}
\end{minipage}

\vspace{5pt}

\begin{minipage}[c]{0.21\linewidth}\centering
\includegraphics[width=3.5cm]{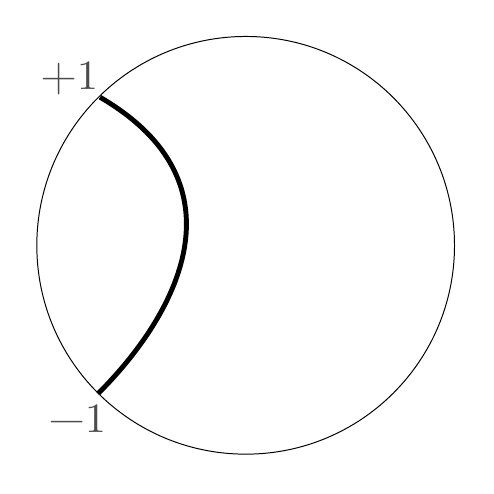}
\end{minipage}
\begin{minipage}[c]{0.21\linewidth}\centering
\begin{tabular}{c}
$\Delta^{1}_{K_{1}}=1$ 
\end{tabular}
\end{minipage}
\begin{minipage}[c]{0.06\linewidth}
\end{minipage}
\begin{minipage}[c]{0.22\linewidth}
\centering
\includegraphics[width=3.5cm]{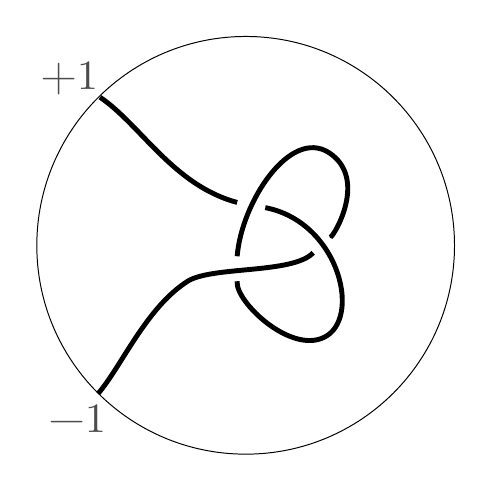}
\end{minipage}
\begin{minipage}[c]{0.23\linewidth}
\centering
\begin{tabular}{c}
$\Delta^{1}_{K_{1} \sharp \overline{T}}=t^2-t+1$ 
\end{tabular}
\end{minipage}

\vspace{5pt}

\begin{minipage}[c]{0.21\linewidth}
\centering
\includegraphics[width=3.5cm]{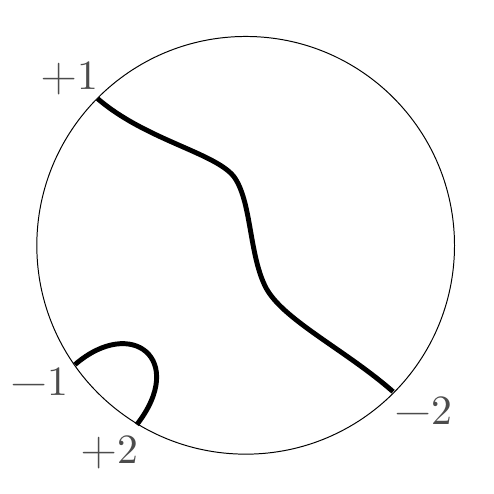}
\end{minipage}
\begin{minipage}[c]{0.21\linewidth}\centering
\begin{tabular}{c}
$\Delta^{1}_{K_{2}}=t+1$ \\
$\Delta^{-1}_{K_{2}}=1$ 
\end{tabular}
\end{minipage}
\begin{minipage}[c]{0.06\linewidth}
\end{minipage}
\begin{minipage}[c]{0.22\linewidth}
\centering
\includegraphics[width=3.5cm]{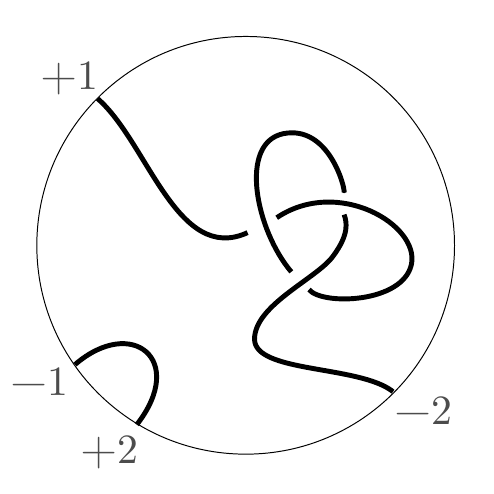}
\end{minipage}
\begin{minipage}[c]{0.23\linewidth}
\centering
\begin{tabular}{c}
$\Delta^{1}_{K_{2} \sharp \overline{T}}=(t+1)(t^4-t^2+1)$ \\
 $\Delta^{-1}_{K_{2} \sharp \overline{T}}=t^4+t^2+1$ 
\end{tabular}
\end{minipage}

\vspace{2pt}
\caption[legenda elenco figure]{Twisted Alexander polynomials for three knots in $L(4,1)$.}\label{tavola2}
\end{figure}

\begin{prop}{\upshape\cite{T}}
Let $L$ be a knot in a lens space then:
\begin{itemize}
\item[1)] $\Delta^{\sigma}_L(t)=\Delta^{\sigma}_L(t^{-1})$ (i.e., the twisted Alexander polynomial is symmetric);
\item[2)] $\Delta(1)=|\textup{Tors}(H_1(L(p,q) \smallsetminus L))|$. 
\end{itemize}
\end{prop}

Before giving the relationship between the twisted Alexander polynomials and the Reidemeister torsion we briefly recall the definition of Reidemeister torsion (for further references see \cite{T}). 

If $c$ and $c'$ are two basis of a finite-dimensional vector space over a field $\mathbb F$, denote with $[c/c']$ the determinant of the matrix whose columns are the coordinates of the elements of $c$ respect to $c'$. Let $C$ be a finite chain complex of vector spaces $$0\to C_m\stackrel{\delta_m}{\to} C_{m-1}\stackrel{\delta_{m-1}}{\to}\cdots\stackrel{\delta_1}{\to}C_0\to 0$$ which is acyclic (i.e., the sequence is exact) and based (i.e., a distinguished  base is fixed for each vector space). 
For each $i\le m$, let $b_{i}$ be a sequence of vectors in $C_{i}$ such that $\delta_{i}(b_i)$ is a base of $\textup{Im}\delta_i$, and let $c_{i}$ be the fixed base of $C_{i}$. The juxtaposition of $\delta_{i+1}(b_{i+1})$ and $b_{i}$ gives a base of $C_{i}$ denoted by $\delta_{i+1}(b_{i+1})b_i$. 
The torsion of $C$ is defined as
$$\tau(C)=\Pi_{i=0}^m  [\delta_{i+1}(b_{i+1})b_i/c_i]^{(-1)^{i+1}}\in \mathbb F.$$
If $C$ is not acyclic the torsion is defined to be zero.

For a finite connected CW-complex $X$, let $\pi=\pi_{1}(X)$ and $H=H_1(X)=\pi/\pi'$.
Consider a ring homomorphism $\varphi:\mathbb Z[H]\to \mathbb F$ and
let $\hat X$ be the maximal abelian  covering of $X$ (corresponding to  $\pi'$).
Let $C_*(\hat X)$ be the cellular chain complex associated to $\hat X$. Since $H$ acts on $\hat X$ via deck transformations, $C_*(\hat X)$  is a complex of left $\mathbb Z[H]$-modules. Moreover the homomorphism $\varphi$ endows $F$ with the structure of  a $\mathbb Z[H]$-module via $fz=f\varphi (z)$, with $f\in F$ and $z\in\mathbb Z[H]$.
Then $\mathbb F \otimes_{\varphi} C_*(\hat X)$ is a chain complex of
finite dimensional vector spaces. The $\varphi$-torsion of $X$ is defined
to be $\tau(\mathbb F \otimes_{\varphi} C_*(\hat X))$. It depends on
the choice of a base for $\mathbb F \otimes_{\varphi} C_*(\hat X)$ and
so the $\varphi$-torsion is defined up to multiplication by $\pm
\varphi(h)$, with $h\in H$.

Let $L$ be a link in $L(p,q)$ and let $X=L(p,q)\smallsetminus L$, then
$X$ is homotopic to  a 2-dimensional cell complex $Y$.   The
$\varphi$-torsion $\tau^{\varphi}_L$ of a link $L$ is the \hbox{$\varphi$-torsion}
of $Y$. In order to
investigate the relationship between the torsion and  the twisted
Alexander polynomial, let $H=\textup{Tors}(H)\times G$ and consider  a map $\sigma:\mathbb Z[H]\to \mathbb
C[G]$ associated to a certain $\sigma\in\hom(\textup{Tors}(H),\mathbb C^*)$, as described in the beginning of this section. If $\mathbb C(G)$ denotes the field of quotient of $\mathbb C[G]$, then by composing with the projection into the quotient, $\sigma$ determines  a homomorphism $\mathbb Z[H]\to \mathbb
C(G)$ that we still denote with $\sigma$.   In this  way each $\sigma\in\hom(\textup{Tors}(H),\mathbb C^*)$  determines both a twisted Alexander polynomial $\Delta^{\sigma}_L$ and a torsion $\tau^{\sigma}_L$.

We say that a link $L\subset L(p,q)$ is \textit{nontorsion} if $\textup{Tors}(H_1(L(p,q) \smallsetminus L))=0$, otherwise we say that $L$ is \textit{torsion}. Note that a local link $L$ in a lens space different from $\s3$ is clearly torsion.
\begin{teo}
Let $L$ be a link in $L(p,q)$. If $L$ is a nontorsion knot and $t$ is a generator of its first homology group, then $\tau^{\sigma}_L(t-1)=\Delta^{\sigma}_L$. Otherwise $\tau^{\sigma}_L(t)=\Delta^{\sigma}_L$.
\end{teo}

\begin{proof}
According to Theorem~\ref{lpqio} and Remark~\ref{ridotto}, the group $\pi_1(L(p,q) \smallsetminus L)$ admits a presentation with $m$ generators and $m-1$ relations. So,  the Alexander-Fox matrix $A$ associated to such a presentation  is a \mbox{$(m-1)\times m$} matrix. This means that $\Delta^{\sigma}(L)=\gcd(\sigma(A_1),\ldots, \sigma(A_m))$, where $A_i$ is the $(m-1)$-minor of $A$ obtained removing the $i$-th column. Let $a_i$ be a generator of $\pi_1(L(p,q) \smallsetminus L)$. The formula $(\sigma(a_i)-1)\tau^{\sigma}_L=\det A_i$  that holds for links in the projective space (see \cite{HL}) generalizes to lens spaces. So, in order to get the statement it is enough to prove that  $\gcd(\sigma(a_1)-1,\ldots,\sigma(a_m)-1)$ is equal to $t-1$, where $t$ is a generator of the free part of $H_1(L(p,q) \smallsetminus L)$, if $L$ is a torsion knot, and equal to 1 otherwise.

Let  $L$ be a torsion knot and denote with  $t$ and $u$ a generator of the free part and the torsion part of $H_1(L(p,q) \smallsetminus L)$ respectively. Moreover let $d$ be the order of the torsion part of $H_1(L(p,q) \smallsetminus L)$. 
If $\textup{pr}(a_i)=t^{h_i}u^{n_i}$ then $\sigma(a_i)=t^{h_i}\zeta^{n_i}$ where $\zeta$ is a $d$-th root of the identity. A simple computation shows that $g$ divides $t^{\sum_{i=1}^{m}h_i}\zeta^{\sum_{i=1}^{m}n_i}-1$, for any $\alpha_i\in\mathbb Z$,  where $g=\gcd(\sigma(a_1)-1,\ldots,\sigma(a_m)-1)$. Since $t\in \textup{pr}(\pi_1(L(p,q) \smallsetminus L))$, there exist $\alpha_i$ such that 
$t=\Pi_{i=1}^m \textup{pr}(a_i^{\alpha_i})=t^{\sum_{i=1}^{m}\alpha_i h_i}u^{\sum_{i=1}^{m}\alpha_i n_i}$; 
so $\sum_{i=1}^{m}\alpha_i h_i=1$ and  $d$ divides $\sum_{i=1}^{m}\alpha_i n_i$. Then $g$ divides $t-1$ and therefore either $g=1$ or $g=t-1$. Analogously, since $u\in \textup{pr}(\pi_1(L(p,q) \smallsetminus L))$,  there exists   $i_0$ such that $g$ divides $\sigma(a_{i_0})-1=t^{h_{i_0}}\zeta^{n_{i_0}}-1$ and $n_{i_0}$ is not divided by $d$. The statement follows by observing that, in this case, $\gcd (t-1,  t^{h_{i_0}} \zeta^{n_{i_0}}-1)=1$. 

If $L$ is torsion and has at least two component then $\sigma(a_i)=t_1^{h_{11}}\cdots t_{\nu}^{h_{1\nu}}\zeta^{n_i}$, where $\nu$ is the number of components. The statement is obtained by setting $t_2=\cdots =t_{\nu}=1$ and applying the previous argument to $t_1$. 

If $L$ is a nontorsion knot, then $H_1(L(p,q) \smallsetminus L)=\langle t\rangle$ and $\sigma(a_i)=t^{h_i}$. In this case it is easy to prove that $\gcd(t^{h_1}-1, \ldots,t^{h_m}-1)=t-1$.

Finally, if  $L$ is nontorsion and has at least two component, then $\sigma(a_i)=t_1^{h_{11}}\cdots t_{\nu}^{h_{1\nu}}$. By letting $t_j=1$ for $j\ne i$ and applying the previous reasoning to $t_i$, for each $i=1,\ldots,\nu$, we obtain $\gcd(\sigma(a_1)-1,\ldots,\sigma(a_m)-1)=\gcd(t_1-~1,\ldots,t_{\nu}-1)=1$. 
\end{proof}

These results generalize those obtained in \cite{K} for knots in $\s3$ and \cite{HL} for link in $L(2,1)\cong\mathbb{RP}^3$. Moreover, in \cite{KL} an analogous result is obtained for CW-complexes but considering only a one-variable Alexander polynomial associated to an infinite cyclic covering of the complex. \\

If $L$ has at least two components we can consider the projection  \\\mbox{$\varphi:\mathbb Z[\zeta][G]=\mathbb Z[\zeta][t_{1}, \ldots , t_{m}, t_{1}^{-1}, \ldots, t_{m}^{-1}] \to\mathbb Z[\zeta][t,t^{-1}]$}, sending each variable $t_{i}$ to $t$. The  \textit{one-variable} twisted Alexander polynomial of $L$ is $\bar{\Delta}^{\sigma}_L=\varphi(\Delta^{\sigma}_L)$.   

The same argument used in the previous proof leads to the following statement, regarding the one-variable twisted polynomial.

\vbox{
\begin{teo}
Let $L$ be a link in $L(p,q)$ with at least two components. If $L$ is a nontorsion link and $t$ is a generator of its first homology group then $\tau^{\sigma}_L(t-1)=\bar{\Delta}^{\sigma}_L$. Otherwise $\tau^{\sigma}_L(t)=\bar{\Delta}^{\sigma}_L$.
\end{teo}
}

The computation of $\bar{\Delta}^{\sigma}_L$ for knots in arbitrary lens spaces has been implemented in a program using  Mathematica code: the input is a knot diagram in $L(p,q)$ given via a generalization of the Dowker-Thistlewaithe code (see \cite{DT,DH,Ta}).

\end{section}

%=========================================================

\vspace{15 pt} {ALESSIA CATTABRIGA, Department of Mathematics,
University of Bologna, ITALY. E-mail: alessia.cattabriga@unibo.it}

\vspace{15 pt} {ENRICO MANFREDI, Department of Mathematics,
University of Bologna, ITALY. E-mail: enrico.manfredi3@unibo.it}

\vspace{15 pt} {MICHELE MULAZZANI, Department of Mathematics and C.I.R.A.M.,
University of Bologna, ITALY. E-mail: michele.mulazzani@unibo.it}

%%%%FINE DEL DOCUMENTO%%%%%%%%%%%%%%%%%%%%%%%%%%%%%%%%%%%%%%%%%%%%%%%%%%%%

\end{document}